\documentclass[12pt,reqno]{amsart}
\usepackage{amsmath,amssymb,amsfonts,amsthm,amstext}

\newcommand{\df}{\textbf}
\newtheorem{prop}{Proposition}

\title
[Extra updates can delay mixing]
{Some circumstances where \\ extra updates can delay mixing}
\author{Alexander E. Holroyd}
\address{Microsoft Research, 1 Microsoft Way,
Redmond, WA 98052, USA} \email{holroyd at
microsoft.com} \urladdr{http://research.microsoft.com/~holroyd}
\date{20 January 2011}

\keywords{mixing time, censoring, coloring, transposition walk}

\subjclass[2010]{60J10; 82C20; 05C15}

\begin{document}
\maketitle

\begin{abstract}
Peres and Winkler proved a `censoring' inequality for Glauber dynamics on
monotone spins systems such as the Ising model.  Specifically, if, starting
from a constant-spin configuration, the spins are updated at some sequence of
sites, then inserting another site into this sequence brings the resulting
configuration closer in total variation to the stationary distribution.  We
show by means of simple counterexamples that the analogous statements fail
for Glauber dynamics on proper colorings of a graph, and for lazy
transpositions on permutations, answering two questions of Peres.  It is not
known whether the censoring property holds in other natural settings such as
the Potts model.
\end{abstract}

\section{Introduction}

Peres and Winkler \cite{peres-winkler} proved the following striking and
useful property of Glauber dynamics on the Ising model.  Consider a finite
set of sites with arbitrary ferromagnetic pair interactions, and let $\pi$ be
the associated stationary distribution on spin configurations.  Starting from
a constant spin configuration, apply single-site updates at a finite
deterministic sequence of sites, where each update consists of replacing the
spin at the chosen site with a random spin chosen according to its
conditional law under $\pi$ given all other spins.  This results in a random
configuration.  If an additional site is inserted into the update sequence,
the resulting configuration is no further from $\pi$ in total variation
distance.

The above result has proved to be an invaluable tool in the analysis of
mixing time for the Ising model; see the applications in
\cite{ding-lubetzky-peres,ding-peres,martinelli-sinclair,martinelli-toninelli,
mossel-sly}.  Peres and Winkler prove their result in the more general
setting of {\em monotone} spin systems; that is, those in which the set of
spins is totally ordered and single-site updates stochastically respect this
ordering. The purpose of this note is to demonstrate that analogous
statements fail in two other natural settings: proper colorings and lazy
transpositions.

We first consider colorings.  Let $G=(V,E)$ be a finite simple graph.  A
(proper) \df{$q$-coloring} of $G$ is a map from $V$ to $\{1,\ldots,q\}$ that
assigns distinct values (colors) to adjacent vertices. Let $\mu$ be a
probability measure on the set of $q$-colorings and let $v\in V$ be a vertex.
Define the \df{recoloring operator} $\kappa(v)$ as follows. Let
$\mu\cdot\kappa(v)$ be the law of the coloring obtaining from a coloring with
law $\mu$ by replacing the color at $v$ with a uniformly random color from
the set of colors absent from $v$'s neighbours (conditional on the existing
coloring). Let $\pi$ be the uniform measure on all $q$-colorings of $G$. An
equivalent interpretation of $\kappa(v)$ is that the color at $v$ is replaced
with a random color chosen according to its conditional law under $\pi$,
given the existing coloring of $V\setminus\{v\}$.  Observe also that $\pi$ is
the stationary distribution of the Markov chain that recolors a random vertex
(chosen according to any distribution with full support) at each step.

\begin{prop}\label{col}
Consider proper $4$-colorings of the triangle. Let $d$ be any continuous
metric on the space of probability measures on colorings, and let $\delta$ be
the point measure on some fixed coloring.  There exist integers $1\leq t<m$
and vertices $i_1,\ldots,i_m,I$ such that, writing
$$\mu=\delta\cdot \kappa(i_1)\cdots \kappa(i_m),$$
$$\nu=\delta\cdot \kappa(i_1)\cdots \kappa(i_t)\kappa(I)
\kappa(i_{t+1})\cdots \kappa(i_m),$$ we have
$$d(\mu,\pi)<d(\nu,\pi).$$
\end{prop}

In other words, if, starting from a deterministic coloring, a sequence of
vertices is recolored, then inserting an extra vertex in the sequence can
move the resulting distribution further from $\pi$.  In particular, the
result applies when $d$ is total variation distance, in which case it answers
a question of Yuval Peres (personal communication).

Now we turn to permutations.  Let $\mu$ be a probability measure on the
symmetric group of permutations of $V:=\{1,\ldots,n\}$.  For a pair $i,j\in
V$, we define the lazy transposition operator $\tau(i,j)$ as follows. Let
$\rho=(\rho(1),\ldots,\rho(n))$ be a random permutation with law $\mu$, and
let $\mu \cdot \tau(i,j)$ be the law of the random permutation obtained from
$\rho$ by interchanging $\rho(i)$ and $\rho(j)$ with probability $1/2$
(conditional on $\rho$), and otherwise leaving $\rho$ unchanged. Let $\delta$
be the point measure on the identity permutation, and let $\pi$ be the
uniform probability measure on all $n!$ permutations.

\begin{prop}\label{perm}
Consider lazy transpositions on $V=\{1,\ldots,4\}$.  Let $d$ be any
continuous metric on the space of probability measures on permutations. There
exist $m,t,I,J,i_1,j_1,\ldots$ such that, writing
$$\mu:=\delta \cdot \tau(i_1,j_1) \cdots \tau(i_m,j_m),$$
$$\nu:=\delta \cdot \tau(i_1,j_1) \cdots \tau(i_t,j_t)
\tau(I,J) \tau(i_{t+1},j_{t+1}) \cdots \tau(i_m,j_m),$$
we have
$$d(\mu,\pi)<d(\nu,\pi).$$
\end{prop}

\section{Proofs}

We will first prove Proposition \ref{perm}, and then deduce Proposition
\ref{col}.

\begin{proof}
We will show that, as $(M,N)\to\infty$,
\begin{equation}\label{pi}
\delta\cdot \big[\tau(2,4)\tau(3,4)\big]^M \tau(1,4)\tau(2,4)
\big[\tau(1,4)\tau(3,4)\big]^N \to \pi,
\end{equation}
while
\begin{equation}\label{not-pi}
\delta\cdot \big[\tau(2,4)\tau(3,4)\big]^M
\tau(1,4)\tau(3,4)\tau(2,4) \big[\tau(1,4)\tau(3,4)\big]^N
\to \alpha,
\end{equation}
for some $\alpha\neq\pi$.  By the continuity of $d$, the
required inequality then follows by taking $M$ and $N$
sufficiently large.

We interpret permutations as arrangements of particles, so that in
permutation $\mu$, particle $\mu(i)$ is in location $i$, and $\tau(i,j)$
swaps the particles in locations $i,j$ with probability $1/2$. Let $\pi_1$ be
the uniform measure on permutations $\rho$ such that $\rho(1)=1$, and note
that
$$
\delta\cdot \big[\tau(2,4)\tau(3,4)\big]^M\to \pi_1\quad\text{as}\quad M\to\infty
$$
(by the convergence theorem for irreducible aperiodic Markov
chains). Now consider a random permutation $\sigma$ with law
$$\beta:=\pi_1 \cdot \tau(1,4)\tau(2,4).$$
The location $\sigma^{-1}(1)$ of particle $1$ is equal to $2$ with
probability $1/4$, since after the first transposition $\tau(1,4)$ it was $1$
or $4$ each with probability $1/2$. Conditional on the location of particle
$1$, the arrangement of particles $2,3,4$ is still uniform, so $\sigma(2)$
(the particle in location $2$) is exactly uniform among $1,\ldots,4$.
Therefore,
$$
\beta \cdot \big[\tau(1,4)\tau(3,4)\big]^N\to \pi\quad\text{as}\quad N\to\infty,
$$
since conditional on $\sigma(2)$, the effect of the additional transpositions
is to uniformize the particles in locations $1,3,4$ in the limit.  The
convergence \eqref{pi} now follows by the continuity of the transposition
operator $\tau(i,j)$.

A similar argument gives \eqref{not-pi}: after applying the extra
transposition $\tau(3,4)$, particle $1$ is at location $4$ with probability
$1/4$ (and cannot be at $2$), therefore after $\tau(2,4)$ it is at $2$ with
probability $1/8$. Thus \eqref{not-pi} holds with $\alpha$ the law of some
random permutation that has $1$ in location $2$ with probability $1/8$.
\end{proof}

\begin{proof}[Proof of Proposition \ref{col}]
Let the triangle $G$ have vertices $1,2,3$, and assume without
loss of generality that $\delta$ is the point measure on the
identity map.  We may identify a $4$-coloring of $G$ with a
permutation assigning colors $1,\ldots, 4$ to {\em four}
vertices $1,\dots,4$, with the color at vertex $4$ being the
one absent from the coloring of $G$.  For $i=1,2,3$, the
operator $\kappa(i)$ corresponds to the lazy transposition
operator $\tau(i,4)$.  Since the example constructed in the
proof of Proposition \ref{perm} uses only transpositions
involving $4$, the same example applies here.
\end{proof}

\section{Further remarks}

The example in the proof of Proposition \ref{perm} was chosen to minimize
computations and facilitate the proof of Proposition \ref{col}. Naturally,
many variations are possible. If $d$ is total variation distance, an explicit
computation shows that the required inequality in fact holds with $M=N=1$. As
a simpler alternative which does not adapt so readily to coloring, we have
$$
\delta\cdot \big[\tau(2,4)\tau(3,4)\big]^M \tau(1,4)\tau(2,4)\tau(1,3)
\to \pi, \quad\text{as}\quad M\to\infty,
$$
while the insertion of $\tau(3,4)$ before the last $\tau(1,3)$ again gives a
different limit.  Finally, if we relax the problem by allowing a ``block
update" $\tau(S)$ (defined so as to uniformly permute the elements of a set
$S\subset V$), then we may of course do away with limits, replacing the
expressions $[\;\;]^M$ and $[\;\;]^N$ with $\tau(\{2,3,4)\}$ and
$\tau(\{1,3,4\})$.

We note that our example adapts to the {\em anti-ferromagnetic} Potts model.
Consider the $4$-state Potts model on a triangle, with anti-ferromagnetic
interactions (i.e.\ favoring distinct spins) of equal strength $J$ along each
edge.  As $J\to \infty$, the transition probabilities for site updates
approach those of the $4$-coloring model.  Hence, starting from a
configuration where all $3$ vertices have different spins, the example in the
proof of Proposition \ref{col} applies here if $J$ is large enough.
Moreover, starting from a constant all-$1$ configuration and updating
vertices $2$ and $3$ results in a configuration that is asymptotically (as
$J\to\infty$) uniform on those where vertex $1$ has spin $1$.  Hence the same
example applies with $M=1$.

It is an open question whether extra updates can delay mixing for the
ferromagnetic Potts model starting from a constant spin configuration.

\bibliographystyle{habbrv}
\bibliography{bib}

\end{document}